\documentclass[12pt]{amsart}
\usepackage{amssymb,latexsym, amsmath, amscd, array, graphicx}

\swapnumbers
\numberwithin{equation}{section}

\theoremstyle{plain}

\newtheorem{thm}{Theorem}[section]

\newtheorem{prop}[thm]{Proposition}

\newcommand\propref{Proposition~\ref}

\theoremstyle{definition}
\newtheorem{defin}[thm]{Definition}

\newtheorem{rems}[thm]{Remarks}

\def\Ker{\protect\operatorname{Ker}}
\def\Im{\protect\operatorname{Im}}

\def\id{\protect\operatorname{id}}

\def\Sv{\protect{\mathfrak{genus}}}
\def\cat{\protect\operatorname{cat}}
\def\cupr{\smallsmile}
\def\TC{\protect\operatorname{TC}}

\def\eps{\varepsilon}
\def\gf{\varphi}
\def\ga{\alpha}

\def\gs{\sigma}

\def\Z{{\mathbb Z}}

\def\R{{\mathbb R}}
\def\N{{\mathbb N}}
\def\F{{\mathbb F}}

\def\1{\hbox{\rm\rlap {1}\hskip.03in{\rom I}}}
\def\Bbbone{{\rm1\mathchoice{\kern-0.25em}{\kern-0.25em}
{\kern-0.2em}{\kern-0.2em}I}}

\def\p{\medskip{\parindent 0pt \it Proof.\ }}

\def\wh{\widehat}

\def\m{\medskip}
\def\sm{\smallskip}

\long\def\forget#1\forgotten{} %

\begin{document}

\title{On higher analogs of topological complexity}
\author[Yu.~Rudyak]{Yuli B. Rudyak}

\address{Yuli B. Rudyak, Department of Mathematics, University
of Florida, 358 Little Hall, Gainesville, FL 32611-8105, USA}
\email{rudyak@ufl.edu}


\begin{abstract}
Farber introduced a notion of topological complexity $\TC(X)$ that is related to robotics. Here we introduce a series of numerical invariants $\TC_n(X), n=2,3,  \ldots$ such that $\TC_2(X)=\TC(X)$ and $\TC_n(X)\le \TC_{n+1}(X)$.  For these higher complexities, we define their symmetric versions that can also be regarded as higher analogs of the symmetric topological complexity.
\end{abstract}

\maketitle

\section{Introduction}

In \cite{F1} Farber introduced a notion of topological complexity $\TC(X)$ and  related it to a problem of robot motion planning algorithm. Here we introduce a series of numerical invariants $\TC_n(X), n=2,3  \ldots$ such that $\TC_2(X)=\TC(X)$ and $\TC_n(X)\le \TC_{n+1}(X)$.  We learn some properties of $\TC_n$ and, in particular, compute $\TC_n(S^k)$. We also define symmetric analogs of higher complexities (=higher analogs of symmetric complexity) introduced in~\cite[Section 31]{F2} and developed in~\cite{FG, GL}.

\m Throughout the paper $\cat X$ denotes the Lusternik--Schnirelmann category of a space $X$, i.e. cat $X$ is one less than the minimal of open and contractible sets in $X$ that cover $X$. For example, $X$ is contractible iff $\cat X=0$.

\m I am grateful to Mark Grant, Jes\'us Gonz\'alez and Peter Landweber who have read the previous versions of the paper and made several useful and helpful comments.

\section{The Schwarz genus of a map}

Given a map $f: X \to Y$ with $X,Y$ path connected, a {\em fibrational substitute} of $f$ is defined as a fibration $\wh f: E \to Y$ such that there exists a  commutative diagram
$$
\CD
X @> h>> E\\
@VfVV @VV \wh f V\\
Y @= Y
\endCD
$$
where $h$ is a homotopy equivalence. The well-known result of Serre~\cite{S} tells us that every map has a fibrational substitute, and it can be proved that any two fibrational substitutes of a map are fiber homotopy equivalent fibrations.

\sm Given a map $f: X \to Y$, we say that a subset $A$ of $Y$ is a local $f$-section if there exists a map $s: A \to X$ (a local section) such that $fs=\id$.

\m The Schwarz genus of a fibration $p: E \to B$ is defined as a minimum number $k$ such that there exists an open covering $U_1, \ldots, U_k$ of $B$ where each map $U_i$ has a local $p$-section, \cite{Sv}. We define the Schwarz genus of a {\it map} $f$ as the Schwarz genus of its fibrational substitute, and we denote it by $\Sv(f)$. This notion is well-defined since any two fibrational substitutes of a map are fiber homotopy equivalent.

\begin{prop}\label{p:comp} For any diagram $X\stackrel{f}\to Y\stackrel{g}\to Z$ we have $\Sv(gf)\ge \Sv(g)$.
\end{prop}

\p This is clear if both $f$ and $g$ (and therefore $gf$) are fibrations. In the general case, replace $f$ and $g$ by fibrational substitutes.
\qed

\m The following remark is useful for applications.

\begin{prop}\label{p:enr} Let $p: E \to B$ be a fibration over a polyhedron $B$. Suppose that $B=X_1\cup \cdots \cup X_n$ where each $X_i$ is an ENR and has a local $p$-section. Then $\Sv(f) \le n$.
\end{prop}

\begin{proof}
We enlarge each $X_i$ to an open subset of $B$ over
which there is a section of $p$.
Take an ENR $X_i=X$ an embedding $X\subset B\subset \R^N$. Let $r: V \to X$ be a neighborhood retraction. Then there exists an open set $U$ of $V$ with $X\subset U\subset V$ such that the maps $U\subset V$ and $U\subset V \stackrel r\to  X \subset V$ are homotopic,~\cite[Chapter 4, especially 8.6, 8.7]{D}. So, there is a homotopy $H: U \times I \to V$, $H(u,0)=u, H(u,1)\subset X$. Consider a section $s: X \to E$ and put $g: U\to E, g(u)=sH(u,1)$. Now use the homotopy extension property to construct a homotopy $G:U\times I \to E$ with $pG=H$ and $G(u,1)=g(u)$. Then $\gs: U \to E$, $\gs(u)=G(u,0)$ is a section over $U$.
\end{proof}

\section{Higher topological complexity}

\begin{defin}\label{d:tc}
Let $J_n, n\in \N$ denote the wedge of $n$ closed intervals $[0,1]_i, i=1, \ldots n$ where the zero points $0_i\in[0,1]_i$ are identified. Consider a path connected space $X$ and set $T_n(X):=X^{J_n}$. There is an obvious map (fibration) $e_n: T_n(X)\to X^n, e_n(f)=(f(1_1), \ldots, f(1_n))$ where $1_i$ is the unit in $[0,1]_i$, and we define  $\TC_n(X)$ to be the Schwarz genus of $e_{n}$.
\end{defin}

\begin{rems}\label{r:big} 1. The above definition makes also sense for $\TC_1(X)$ (to be always equal to 1), but the notation that are started from $\TC_n, n>1$ turns out to be more elegant.

2. It is easy to see that $\TC_n(X)\ge \TC_n(Y)$ if $X$ dominates $Y$. So, $\TC_n$ is a homotopy invariant.

3. It is also worth noting that the fibration $e_n$  can be described as follows: Take the diagonal map $d_n: X \to X^n$ and regard $e_n$ as its fibrational substitute \`a la Serre. Hence, in fact, the higher topological complexity $\TC_n(X)$ is the Schwarz genus of the diagonal map $d_n: X \to X^n$. Note also that the (homotopy) fiber of $e_n$ is $(\Omega X)^{n-1}$ where $\Omega X$ denotes the loop space of $X$.

4. The fibration $e_n$ is homotopy equivalent to the following fibration $e'_n$. Define $S_n(X)\subset X^I \times X^n$ as
\[
S_n(X)=\{(\ga, x_1, \ldots, x_n)\bigm | x_i\in \Im(\ga: I \to X, i=1, \ldots, n)\}
\]
and define $e'_n: S_n(X)\to X^n$ as $e'_n(\ga, x_1, \ldots, x_n)=(x_1, \ldots, x_n)$. To prove that $e'_n$ is a fibrational substitute of $d_n$,  consider the homotopy equivalence $h : X \to S_n(X), h(x)=(\eps_x, x, \ldots, x)$ where $\eps_x$ is the constant path at $x$. Note that $e'_nh=d_n: X \to X^n$, and thus $e'_n$ is the fibrational substitute of $d_n$.

5. The fibration $e_n$ is homotopy equivalent to the fibration
\[
e''_n: X^I \to X^n,
 e''_n(\ga)=\left(\ga(0), \ga\left(\frac 1 {n-1}\right), \ldots, \ga\left(\frac{k}{n-1}\right), \ldots, \ga(1)\right)
\]
where $\ga: I \to X$. Indeed, consider the homotopy equivalence $h : X \to X^I, h(x)=\eps_x$, and note that $e''_nh=d_n$.

6. It is easy to see (especially in view of the previous item) that $\TC_2(X)$ coincides with the topological complexity $\TC(X)$ introduced by Farber~\cite{F1}.

7. Mark Grant pointed out to me that, similarly to $\TC_2(X)$, the invariant $\TC_n(X)$ is related to robotics. In detail, $\TC_2(X)$ is related to motion planning algorithm when a robot moves from a point to another point, while $\TC_n(X)$ is related to motion planning problem whose input is not only an initial and final point but also an additional $n-2$ intermediate points.
\end{rems}

\m
\begin{prop}\label{p:ineq} $\TC_n(X)\le \TC_{n+1}(X)$.
\end{prop}

\p Let $d_k: X \to X^k$ denote the diagonal, $d_k(x)=(x, \ldots, x)$. Note that $\TC_k(X)$ is the Schwarz genus of the map $d_k$. Define
$$
\gf: X^n \to X^{n+1}, \gf(x_1, \ldots, x_{n-1},x_n)=(x_1, \ldots, x_{n-1}, x_n, x_n).
 $$
 Then $d_{n+1}=\gf d_n$, and hence the Schwarz genus of $d_{n+1}$ is greater than or equal to the Schwarz genus of $d_{n}$ by \propref{p:comp}.
 \qed

\m To compute $\TC_n$, we can apply known methods of calculation of the Schwarz genus. For example, the Schwarz genus of a fibration over $B$ does not exceed $1+\cat B$. So,
\begin{equation}\label{eq:cat}
\TC_n(X)\le 1+\cat (X^n)\le n\cat X+1.
\end{equation}
Furthermore, we have the following claim,~\cite[Theorem 4]{Sv} (here, generally, $H^*(X; A_i)$ denotes cohomology with twisted coefficients).

\begin{prop}\label{p:diag} Let $d_n: X \to X^n$ be the diagonal. If there exist $u_i\in H^*(X^n;A_i), i=1, \ldots, m$ so that $d_n^*u_i=0$ and
$$
u_1\cupr \cdots \cupr u_m\ne 0\in H^*(X^n;A_1\otimes\cdots \otimes A_m),
$$
 then $\TC_n(X)\ge m+1$.
\qed
\end{prop}

\begin{prop}\label{p:est} If $X$ is a connected finite CW-space that is not contractible, then $\TC_n(X)\ge n$ .
\end{prop}

\p If $X$ is $(k-1)$-connected with $k>1$ then $H^k(X;\F)\ne 0$ for some field $\F$. Take a non-zero $v\in H^k(X;\F)$ and put $v_i=p_i^*v$ where $p_i: X^n \to X$ is the projection onto the $i$th factor. Then $u_i:=v_i-v_{n}\in \Ker d_n^*$ for $i=1,\ldots, n-1$ and $u_1\cupr \cdots \cupr u_{n-1}\ne 0$, and so $\TC_n(X)\ge n$ by \propref{p:diag}.

\m Now, assume that $X$ is not simply connected. Take the Berstein class $v \in H^1(X; I)$ where $I$ is the augmentation ideal in the integral group ring of $\pi_1(X)$, see~\cite{B,DR}. Then argue as in the
previous paragraph.
\qed

\section{An Example: $\TC_n(S^k)$}

Farber~\cite[Theorem 8]{F1} proved that $\TC(S^k)=2$ for $k$ odd and $\TC(S^k)=3$ for $k$ even. We extend this result (and method) and show that $\TC_n(S^k)=n$ for $k$ odd and $\TC_n(S^k)=n+1$ for $k$ even. Fix $n>2$ and $k>0$.

\m  For $k$ even, take a generator $u\in H^k(S^k)=\Z$ and denote by $u_i$ its image in the copy $S^k_i$ of $S^k$, $i=1, \ldots, n$. In the class $H^k((S^k)^n)$, consider the element
$$
v=\left(\sum_{i=1}^{n-1}1\otimes \cdots \otimes 1 \otimes u_i \otimes 1 \otimes \cdots \otimes 1\right)-1\otimes \cdots \otimes 1 \otimes (n-1)u_n.
$$
Then $v^n=(1-n)n!(u_1\otimes \cdots \otimes u_n)$ since $k$ is even, and so $v^n\ne 0$. On the other hand, $d_n^*v=0$. Thus, $\TC_n(S^k)=n+1$ by \eqref{eq:cat} and \propref{p:diag}.

\m Now we prove that $\TC_n(S^k)=n$ for $k$ odd. Consider a unit tangent vector field $V$ on $S^k$, $V=\{V_x \bigm| x\in S^k\}$. Given $x,y\in S^k$ such that $y$ is the antipode of $x$, denote by $[x,y]$ the path $[0,1]$ determined by the geodesic semicircle joining $x$ to $y$ and such that the $V_x$ is the direction of the semicircle at $x$.

\sm Furthermore,  if $x$ and $y$ are not antipodes, denote by $[x,y]$ the path $[0,1]$ determined by the shortest geodesic from $x$ to $y$.

\sm Define an injective (non-continuous) function
\begin{equation*}
\begin{aligned}
&\gf: (S^k)^n \longrightarrow T_n(S^k), \\
&\gf(x_1, \ldots, x_n)=\{[x_1,x_1], \ldots, [x_1, x_n]\}.
\end{aligned}
\end{equation*}

\m For each $j=0, \ldots, n-1$ consider the submanfold (with boundary)
$U_j$ in $(S^k)^n$ such that each $n$-tuple $(x_1, \ldots, x_n)$ in $U_j$ has exactly $j$ antipodes to $x_1$. Then $\gf|_{U_j}: U_j\to T_n(S^k)$ is a continuous section of $e_n$, and $\bigcup_{i=0}^{n-1}U_i=(S^k)^n$. Furthermore, each $U_i, i=0, \ldots, n-1$ is an ENR, and so $\TC_n(S^k) \le n$ by \propref{p:enr}. Thus, $\TC_n(S^k)=n$ by \propref{p:est}.

\section{Sequences $\{\TC_n(X)\}$}

Of course, it is useful and interesting to compute invariants $\TC_n(X)$ for different spaces.

\sm However, there is a general problem: to describe all possible (non-decreasing) sequences that can be realized as $\{\TC_n(X)\}_{n=1}^\infty$ with some fixed $X$.

\sm As a first step, note that the inequality $\TC(X)\ge 1+\cat X$
(\cite[Proposition 4.19]{F3})
together with \eqref{eq:cat} imply that
\begin{equation}\label{eq:growth}
\TC_n(X) \le n\TC_2(X)-n+1.
\end{equation}
So, any sequence $\{\TC_n(X)\}$ has linear growth.

\m Given $a\in \N$, we can also consider two functions
$$
f_a(n)=\max_X\{\TC_n(X)\bigm | \TC(X)=a\}
$$
and
$$
g_a(n)=\min_X\{\TC_n(X)\bigm | \TC(X)=a\}.
$$
 So,
\begin{equation}
n\le g_a(n)\le f_a(n)\le na-n+1
\end{equation}
We can ask about the evaluation of the functions $f_a$ and $g_a$. (This question was inspired by a discussion with M. Grant.)

\sm Now we show that $g_3(n)<f_3(n)$ for $n>2$.

\sm We have $\TC(S^2)=3=\TC(T^2)$ (here $T^2$ is the 2-torus, the last equality can be found in~\cite[Theorem 13]{F1}).

\begin{prop}\label{p:torus} $\TC_n(T^2)\ge 2n-1$.
\end{prop}

\begin{proof}
Let $x,y$ be the canonical generators of $H^1(T^2)$.
Put $x_i=p_i^*x$ where $p_i: (T^2)^n \to T^2$ is the projection on $i$th factor.
Similarly, put $y_i=p^*y$. Then $d_n^*(x_2-x_i)=0=d_n^*(y_2-y_i)$ for $i=2, \ldots, n$.
On the other hand, the product
$$
(x_2-x_1)\cupr \cdots \cupr (x_n-x_1)\cupr (y_2-y_1)\cupr \cdots \cupr (y_n-y_1)
$$
is non-zero. Indeed, it maps to $x_2\cupr \cdots \cupr x_n \cupr y_2\cupr \cdots \cupr y_n \ne 0$ under the inclusion $(T^2)^{(n-1)} \to (T^2)^n$ on the last $n-1$ copies of $T^2$.

\sm
Now the claim follows from \propref{p:diag}.
\end{proof}

Thus, for $n>2$ we have
$$
g_3(n)\le \TC_n(S^2)=n+1<2n-1 \le \TC_n(T^2)\le f_3(n).
$$

\m So, we see that the sequence $\{\TC_n(X)\}$ contains more information on (the complexity of) a space $X$ than just the number $\TC(X)$.

\section{Symmetric topological complexity}

Farber~\cite[Section 31]{F2} considered a symmetric version $\TC^S$(X) of the topological complexity. More detailed information about this invariant can be found in the papers Farber--Grant~\cite{FG} and Gonz\'alez--Landweber~\cite{GL}. We define its higher analogs $\TC_n^S(X)$ as follows:  Let $\Delta=\Delta^n_X \subset X^n$ be the discriminant,
$$
\Delta=\{(x_1, \ldots, x_n)\bigm| x_i=x_j \text{ for some pair $(i,j)$ with $i\ne j$ }\}
$$
The space $X^n\setminus \Delta$ consists of ordered configurations of $n$ distinct points in $X$ and is frequently denoted by $F(X,n)$. Let $v_n: Y \to F(X,n)$ be the restriction of the fibration $e_n$. Then the symmetric group $\Sigma_n$ acts on $Y$ by permuting paths and on $F(X,n)$ by permuting coordinates. These actions are free and the map $v_n$ is equivariant. So, the map $v_n$ yields a map (fibration) $ev_n$ of the corresponding orbit spaces, and we define $\TC_n^S(X)$ as $\TC_n^S(X)=1+\Sv(ev_n)$. Note that, for the symmetric complexity we have $\TC^S(X)=\TC_2^S(X)$.

\m It is worth mentioning that in case $X=\R^2$ the space $F(X,n)/\Sigma_n$ is the classifying space for the $n$-braid group $\beta_n$. So, the symmetric topological complexity $\TC_n^S$ turns out to be related to the topological complexity of algorithms considered by Smale~\cite{Sm} and Vassiliev~\cite{V}.


\begin{thebibliography}{[Sva66]}
\bibliographystyle{amsalpha}

\bibitem[B76]{B} I. Berstein. On the
Lusternik--Schnirelmann category of Grassmannians. \emph
{Math. Proc. Camb. Phil. Soc.} \textbf{79} (1976), no 1,
129--134.

\bibitem[DR09]{DR} Dranishnikov, A., Rudyak, Yu.: On the
Berstein-Schwarz Theorem in dimension 2.  \emph {Math. Proc. Cambridge Phil. Soc.} 146 (2009), no 2, 407--413.

\bibitem[D95]{D} Dold, Albrecht: \emph {Lectures on algebraic topology.} Reprint of the 1972 edition. Classics in Mathematics. Springer-Verlag, Berlin, 1995.

\bibitem[F03]{F1} Farber,M:
Topological complexity of motion planning, \emph{Discrete Comput.Geom.} 29 (2003), 211-221

\bibitem[F06]{F2} Farber,M: Topology of robot motion planning.
\emph{Morse theoretic  methods in nonlinear analysis and in symplectic topology,} 185--230, NATO Sci. Ser. II Math. Phys. Chem., 217, Springer, Dordrecht, 2006.

\bibitem[F08]{F3} Farber,M: \emph{Invitation to topological robotics.}
Zurich Lectures in Advanced Mathematics, EMS, Z\"urich, 2008.

\bibitem[FG07]{FG} Farber, M.; Grant, M.
Symmetric motion planning. \emph{Topology and robotics,} 85--104,
Contemp. Math., \textbf{438}, Amer. Math. Soc., Providence, RI, 2007.


\bibitem[GL09]{GL} Gonz\'alez, J.; Landweber, P.
Symmetric topological complexity of projective and lens spaces. .
\emph{Algebr. Geom. Topol.} \textbf{9} (2009), no. 1, 473--494.


\bibitem[S51]{S} Serre, J.-P. Homologie singul\`ere des espaces fibr\'es. Applications.
\emph{Ann. of Math.} (2) \textbf{54}, (1951). 425--505.

\bibitem[Sm87]{Sm} Smale, S.  On the topology of algorithms. I. \emph{J. Complexity} \textbf{3} (1987)  no. 2, 81--89.

\bibitem[Sva66]{Sv} \v Svarc, A: The genus of a fiber space,
\emph{Amer. Math. Soc. Transl.} Series~2, \textbf{55} (1966), 49--140.

\bibitem[V88]{V} Vassiliev, V: Cohomology of braid groups and complexity of algorithms, \emph{Functional Analysis and its Appl.,} \textbf{22} (1988), 15--24

\end{thebibliography}
\end{document}